\documentclass{amsart}
\usepackage{amsthm,amsmath,amsthm,amssymb,mathabx,mathrsfs,graphicx,mathtools,caption,relsize,hyperref,cleveref,enumerate,nameref,float,chngcntr,xcolor}

\theoremstyle{plain}
 \newtheorem{thm}{Theorem}[section]
 \newtheorem{prop}[thm]{Proposition}
 \newtheorem{lem}[thm]{Lemma}
 \newtheorem{cor}[thm]{Corollary}
\theoremstyle{definition}
 \newtheorem{exm}[thm]{Example}
 \newtheorem{dfn}[thm]{Definition}
 \newtheorem{rem}[thm]{Remark}
 \newtheorem*{claim*}{Claim}
\theoremstyle{remark}
 \numberwithin{equation}{section}
 
\def\PGL{\mbox{\rm PGL}}
\def\Qp{\mathbb{Q}_p}
\def\Zp{\mathbb{Z}_p}

\def\Aut{\text{\rm Aut}}

\def\Ch{\text{\rm Ch}}
\def\proj{\text{\rm proj}}
\def\Fix{\text{\rm Fix}}
\def\id{\text{\rm id}}
\def\Sym{\text{\rm Sym}}
\def\T{{T}}
\def\Tq{{T}_{d}}

\def\TQ{{T}_{\textbf{a}}}

\def\con{\mathrm{con}}
\def\F{F}
\def\FF{F'}

\renewcommand{\le}{\leqslant}\renewcommand{\leq}{\leqslant}
\renewcommand{\ge}{\geqslant}\renewcommand{\geq}{\geqslant}
\renewcommand{\setminus}{\smallsetminus}
\title[Homomorphic Images of Groups Acting on Trees and Buildings]{Homomorphic Images of Locally Compact Groups Acting on Trees and Buildings}
\date{\today}


\author[Max Carter]{Max Carter} 
\address{School of Mathematical and Physical Sciences \\ 
University of Newcastle \\ 
Callaghan \\
Australia}
\email{max.carter.math@gmail.com}

\author[George Willis]{George Willis}
\address{School of Mathematical and Physical Sciences \\ 
University of Newcastle \\ 
Callaghan \\
Australia}
\email{george.willis@newcastle.edu.au}

\thanks{\textit{Acknowledgements:} The second author acknowledges support from the Australian Research Council grant FL170100032} 

\begin{document}

\begin{abstract}
We study analogues of Cartan decompositions of Lie groups for totally disconnected locally compact groups. It is shown using these decompositions that a large class of totally disconnected locally compact groups acting on trees and buildings have the property that every continuous homomorphic image of the group is closed.
\end{abstract}

\maketitle

\section{Introduction}

Every locally compact group $G$ admits a short exact sequence
\begin{displaymath} \{1\} \rightarrow G_0 \rightarrow G \rightarrow G/G_0 \rightarrow \{1\} \end{displaymath}
where $G_0$ is the connected component of the identity in $G$. The set $G_0$ is a normal connected locally compact subgroup of $G$ and $G/{G_0}$ is a totally disconnected locally compact group. Consequently, the study of the structure of locally compact groups essentially splits into the study of the connected locally compact groups and the totally disconnected locally compact groups.

Connected locally compact groups are already fairly well understood: in work by Gleason, Montgomery and Zippin \cite{Gle51, Gle52, MZ52} to solve Hilbert's fifth problem, connected locally compact groups have been identified as inverse limits of connected Lie groups. Totally disconnected locally compact groups (t.d.l.c.\ groups from now on) on the other hand are not as well understood and are resisting a general structure theory. Until the mid 90's, the only known general result for t.d.l.c.~groups was a theorem by van Dantzig from 1931 \cite{VD31, VD36} which asserts that every t.d.l.c.\ group admits a base of neighbourhoods of the identity made up of compact open subgroups. 

In recent years significant progress has been made in the structure theory of t.d.l.c.~groups, though, a general structure theory is still far from complete. Automorphism groups of locally finite connected graphs have played a fundamental role in the study of t.d.l.c.~groups to date, partially due to the Cayley-Abels graph construction \cite{Abe73, KM08}. The present article focuses on problems concerning these groups. 

We extend work from the paper \cite{CW20} to a more general class of t.d.l.c.~groups acting on trees and buildings. In \textit{loc.\ cit.}, the authors transfer ideas from the study of Lie groups to the study of t.d.l.c.~groups acting on trees: analogues of Cartan decompositions from Lie theory are studied for automorphism groups of label-regular trees. Using the given Cartan decomposition, it is shown that the simple subgroup of the automorphism group of a label-regular tree generated by edge fixators has the property that every continuous homomorphism from the group has closed range \cite[Theorem 4.2]{CW20}. This property is shared with simple Lie groups (\cite{Omo66}, also see \cite{BG14} for historical remarks).

In the present article we abstract an idea seen in the proof of \cite[Theorem 4.2]{CW20} and call it the contraction group property, see Definition \ref{dfn:contraction}. In Theorem \ref{thm:closedrangethm}, it is shown that any topologically simple group satisfying the contraction group property has the closed range property, that is, every continuous homomorphism from the group has closed range. This generalises the result \cite[Theorem 4.2]{CW20}. We use Theorem \ref{thm:closedrangethm} to show that members of a much larger class of t.d.l.c.\ groups acting on trees have the closed range property, see Theorem \ref{thm:mainthm}. We also study the contraction group and closed range properties more generally and prove a few general results. The article concludes by proving an analogue of Theorem \ref{thm:mainthm} for automorphism groups of semi-regular right-angled buildings, see Theorem \ref{thm:mainthm2}.


\section{Preliminaries}\label{sec:prelim}

Here we recall some notation and terminology that will be used throughout the article. Let $V\Gamma$ be a set and $E\Gamma \subseteq \{ \{ u,v \} : u,v \in V\Gamma \}$. The pair $\Gamma = (V\Gamma, E\Gamma)$ is called a \textit{graph}. The elements of $V\Gamma$ are called \textit{vertices} and elements of $E\Gamma$ \textit{edges}. For $v \in V\Gamma$, the cardinality of the set $E(v) := \{ e \in E\Gamma : v \in e \}$ is called the \textit{degree} of $v$, denoted $deg(v)$, and we call $\Gamma$ \textit{locally-finite} if $deg(v)$ is finite for all $v \in V\Gamma$. 

Let $I$ be a subinterval of $\mathbb{Z}$ with its usual order, and let $(v_i)_{i \in I}$ a sequence in $V\Gamma$. The sequence is called a \textit{path} if $v_i \ne v_{i+1}$ and $\{v_i, v_{i+1}\} \in E\Gamma$ for all $i \in I$. If $(v_i)_{i \in I}$ is a path and $I$ has finite cardinality, then we define its \textit{length} to be $\lvert I \rvert -1$ ({\it i.e.\/} the number of edges in the path), otherwise, we say the path is a \textit{ray} if $I = \mathbb{N}$ or \textit{bi-infinite} if $I = \mathbb{Z}$. 

If $(v_i)_{i \in I}$ is a path and $I = \{0, 1, \dots, n\}$, we call the vertices $v_0$ and $v_n$ the \textit{endpoints} of the path, and say that the path is a \textit{cycle} if $v_0 = v_n$. Given two vertices $u,v \in V\Gamma$, the \textit{distance} between $u$ and $v$, denoted $d(u,v)$, is the length of a shortest path with endpoints $u$ and $v$, provided that such a path exists. Define the following two sets: 
$$
B(v,n) := \{ u \in V\T \mid d(v,u) \le n \}\text{ and }S(v,n) := \{ u \in V\T \mid d(u,v)=n \}.
$$

For $v \in V\Gamma$, an edge of the form $\{v,v\}$ is called a \textit{loop}. $\Gamma$ is called \textit{connected} if for any two vertices $v,u \in V\Gamma$, there is a path in $\Gamma$ with endpoints $v$ and $u$. A connected graph with no loops is called \textit{simple}. All graphs will be assumed to be simple. A simple graph with no cycles is called a \textit{tree}. The \textit{regular tree} of degree $d$, also called the \textit{$d$-regular} tree and denoted by $\Tq$, is the infinite tree in which every vertex has degree $d$. \textit{Label-regular trees}, denoted by $\TQ$, will be mentioned in the article; the reader may consult \cite[Section 2]{CW20} for the definition of these if required.

Fix an infinite tree $\T = (V\T, E\T)$. Two rays in $\T$ are said to be equivalent if their intersection is also a ray. It is easily checked that this is an equivalence relation on the set of all rays in $\T$. The set of equivalence classes of rays under this equivalence relation is called the \textit{boundary} of $\T$ and denoted by $\partial \T$. The elements of $\partial \T$ are the \textit{ends} of $\T$.

Suppose that the group $G$ acts on a set $X$ and let $Y \subseteq X$. Then the \textit{stabiliser} of $Y$ under the action of $G$ is $G_{Y}:=\{g \in G \mid g(Y)=Y \}$. If $Y = \{ y \}$ then we will write $G_y$ instead of $G_{\{y\}}$. Similarly, $\text{Fix}_{G}(Y) := \{ g \in G \mid g(y) = y \; \forall y \in Y \}$ and is called the \textit{fixator} of $Y$ in $G$. For $x \in X$, $Gx := \{ g(x) \mid g \in G \}$ denotes the \textit{orbit} of $x$ under $G$. 

$\text{Aut}(\Gamma)$ will denote the group of all graph automorphisms of the graph $\Gamma$. Throughout this paper, it will be assumed that $\text{Aut}(\Gamma)$ has the permutation topology: the topology whose base neighbourhoods are sets of the form 
$$
\mathcal{U}(g_0,\mathcal{F}) := \{ g \in \text{Aut}(\Gamma) : g(x) = g_0(x) \; \text{for all} \; x \in \mathcal{F} \},
$$ 
with $g_0$ ranging over all elements of the group $\text{Aut}(\Gamma)$ and $\mathcal{F}$ ranging over all finite subsets of $V\Gamma$. It is assumed that any subgroup of $\text{Aut}(\Gamma)$ has the induced topology.

Now let $\T$ be a tree and $G$ a group of automorphisms of $\T$. For $e = \{v,w\} \in E\T$ and $k \in \mathbb{N}$ define $F_{k,e} := \Fix_{G}(B(v,k) \cap B(w,k)) = \Fix_{G}(B(v,k-1) \cup B(w,k-1))$. We define the following two groups:
\begin{displaymath} G^{(k)} := \{ g \in \Aut(\T) \mid \forall v \in V\T \; \exists g_o \in G \; \text{such that} \; g|_{B(v,k)} = g_0|_{B(v,k)} \} \end{displaymath}
called the $k$-closure of $G$ and
\begin{displaymath} G^{+_k} := \langle F_{k,e} \mid e \in E\T \rangle. \end{displaymath}

We will also denote by $\T_{(v,w)}$ the half-tree containing the vertex $w$ formed by removing the edge $\{v,w\}$ from $\T$.

The final section of the article discusses buildings. We will for the most part follow the notation and terminology set out in \cite{AB08}. Fix a Coxeter system $(W,S)$ and let $\ell$ be the length function on words in $W$ with respect to the generating set $S$. The Coxeter system is called \textit{spherical} if $W$ is finite and \textit{right-angled} if the off-diagonal entries of the associated Coxeter matrix \cite[Definition 1.95]{AB08} consist of only 2's and $\infty$'s.  

Following \cite[Definition 5.1]{AB08}, a \textit{building} $\Delta$ of \textit{type} $(W,S)$ is a pair $(\Ch(\Delta), \delta)$, where $\Ch(\Delta)$ is a non-empty set whose elements are called the \textit{chambers} of $\Delta$, and $\delta: \Ch(\Delta) \times \Ch(\Delta) \rightarrow W$ a map called the \textit{Weyl distance} function which satisfies the following properties:

\begin{enumerate}[(i)]
   \item $\delta(C,D) = 1$ if and only if $C = D$. \
   \item If $\delta(C,D) = w$ and $C' \in \Ch(\Delta)$ satisfies $\delta(C',C) = s \in S$, then either $\delta(C',D) = w$ or $\delta(C',D) = sw$. If, in addition, $\ell(sw) = \ell(w) + 1$, then $\delta(C',D) = sw$.\
   \item If $\delta(C,D) = w$, then for any $s \in S$ there is a chamber $C' \in \Ch(\Delta)$ such that $\delta(C', C) = s$ and $\delta(C', D) = sw$.\
\end{enumerate} 

Note that such a building can be viewed as an edge-labelled graph: the vertices are the elements of $\Ch(\Delta)$ and two chambers $C, D \in \Ch(\Delta)$ are connected by an edge of label $s \in S$ if $\delta(C,D) = s$ (c.f.\ \cite[Chapter 1]{Wei04}). The automorphism group $\Aut(\Delta)$ can then be given the permutation topology as described above.

Now let $J \subseteq S$. Two chambers $C,D \in \Ch(\Delta)$ are said to be $J$-equivalent, which we denote by $C \sim_{J} D$, if $\delta(C,D) \in W_{J}$ where $W_{J} := \langle J \rangle \le W$. This is an equivalence relation on the set of chambers of $\Delta$. The equivalence classes under this equivalence relation are called \textit{$J$-residues}, and the $J$-residue containing the chamber $C \in \Ch(\Delta)$ will be denoted by $\mathcal{R}_{J}(C)$. An arbitrary subset $\mathcal{R} \subseteq \Ch(\Delta)$ is called a \textit{residue} if it is a $J$-residue for some $J \subseteq S$. 

If $J = \{s\}$ for some $s \in S$, we say that two chambers $C$ and $D$ are \textit{$s$-equivalent} and write $C \sim_{s} D$. Moreover, if $\delta(C,D) = s$ then we say that $C$ and $D$ are \textit{$s$-adjacent}, and two chambers are said to be \textit{adjacent} if they are $s$-adjacent for some $s \in S$. The equivalence classes in $\Ch(\Delta)$ under the equivalence relation $\sim_{s}$ ($s \in S$) are called \textit{$s$-panels}. The term \textit{panel} is used to refer to an $s$-panel for some $s \in S$. The unique $s$-panel containing a chamber $C \in \Ch(\Delta)$ will be denoted by $\mathcal{P}_{s}(C)$. A building such that every panel has cardinality two is called \textit{thin}, and a building where every panel has cardinality strictly greater than two is called \textit{thick}. A thin subbuilding of the building $\Delta$ is called an \textit{apartment} of $\Delta$ and is necessarily isometric to the Cayley graph of $(W,S)$.

A \textit{gallery} of length $n$ in $\Delta$ is a sequence of chambers $\Gamma: C_0, \dots, C_n$ such that $C_{i-1}$ is adjacent to $C_i$ for each $i$. If there is no gallery of shorter length between $C_0$ and $C_n$, then we define the distance $d(C_0,C_n)$ between $C_0$ and $C_n$ to be $n$. One can show that $d(C,D) = \ell(\delta(C,D))$ for any $C,D \in \Ch(\Delta)$. The gallery $\Gamma$ is called \textit{minimal} if $d(C_0,C_n) =n$. 


Given a residue $\mathcal{R}$ and a chamber $D \in \Ch(\Delta)$, define $d(\mathcal{R}, D) := \min\{ d(C,D) \mid C \in \Ch(\mathcal{R}) \}$. It can be shown that there is a unique chamber $C_1 \in \Ch(\mathcal{R})$ such that $d(C_1,D) = d(\mathcal{R}, D)$ \cite[Proposition 5.34]{AB08}. The chamber $C_1$ is then called the \textit{projection} of $D$ onto $\mathcal{R}$ and is denoted by $\proj_{\mathcal{R}}(D)$. For two residues $\mathcal{R}_1$ and $\mathcal{R}_2$, we define $\proj_{\mathcal{R}_1}(\mathcal{R}_2) := \{ \proj_{\mathcal{R}_1}(C) \mid C \in \Ch(\mathcal{R}_2) \}$.

A building is called \textit{right-angled} if the underlying Coxeter system $(W,S)$ is right-angled and \textit{semi-regular} if for every $s \in S$, every $s$-panel has the same cardinality. By \cite[Proposition 1.2]{HP03}, for every sequence $(q_s)_{s \in S}$ of cardinal numbers and right-angled Coxeter system $(W,S)$, there exists a unique semi-regular right-angled building $\Delta$ of type $(W,S)$ such that every $s$-panel has cardinality $q_s$ for all $s \in S$. The reader can consult \cite{DMSS18, Cap14} for more details on semi-regular right-angled buildings. We will refer to these articles whenever necessary.


\section{The Contraction Group and Closed Range Properties}

\subsection{Cartan-like Decompositions}

In this section we extend ideas from the paper \cite{CW20}. The following notion will be used throughout.

\begin{dfn}[Cartan-like Decomposition]\label{dfn:Cartan-like}
Let $G$ be a topological group. A \textit{Cartan-like decomposition} of $G$ is a double coset decomposition $G=KAK$ where $K \le G$ is a compact open subgroup and $A \subseteq G$ is a set of coset representatives for the decomposition.
\end{dfn}

\begin{rem}
Note that we do not assume any structure on the set $A$; it just has to be a subset of the group. This is distinct from a Cartan decomposition of a Lie group, in which $A$ is a subgroup, hence the name ``Cartan-like".
\end{rem}

\begin{dfn}[Contraction Group] 
Let $G$ be a topological group and $(g_i)_{i \in I} \subseteq G$ a sequence. The \textit{contraction group} of the sequence, denoted $\con((g_i)_{i \in I})$, is defined by $\con((g_i)_{i \in I}) := \{ x \in G \mid g_{i}xg_{i}^{-1} \rightarrow \id_G \}$. 
\end{dfn}

Let $\TQ$ be a label-regular tree (see \cite[Section 2]{CW20} for definition). It was shown in the proof of \cite[Theorem 4.2.]{CW20}, that if we take any sequence of coset representatives in the Cartan-like decomposition of $\Aut(\TQ)$ with respect to a maximal compact open subgroup, then either the sequence is bounded or has a subsequence with non-trivial contraction group. This led to the fact that any continuous homomorphism from the simple subgroup $\Aut^{+_1}(\TQ) \le \Aut(\TQ)$ has closed range. From now on, we will use the following terminology, which reflects the ideas discussed in this paragraph.

\begin{dfn}[Contraction Group Property] \label{dfn:contraction}
Let $G$ be a topological group and $G = KAK$ a Cartan-like decomposition of $G$. We say that this decomposition has the \textit{contraction group property} if every sequence of coset representatives ({\it i.e.\/} elements of $A$) is either bounded or has a subsequence with non-trivial contraction group. We say that a group $G$ has the contraction group property if there exists a compact open subgroup $K \le G$ and a Cartan-like decomposition of $G$ with respect to $K$ satisfying the contraction group property.
\end{dfn}

And we define the closed range property as follows.

\begin{dfn}
[Closed Range Property] A topological group $G$ has the \textit{closed range property} if every continuous homomorphism $\varphi: G \rightarrow H$, where $H$ is an arbitrary topological group, has closed range {\it i.e.\/} $\varphi(G)$ is closed in $H$.
\end{dfn}


\subsection{The Contraction Group Property}  

In this section we will prove some general results concerning the contraction group property. First we show that the contraction group property is independent of the choice of compact open subgroup.

\begin{prop}\label{prop:cosdependence}
Let $G$ be a topological group. Suppose that there exists a compact open subgroup $K \le G$ and a Cartan-like decomposition $G=KAK$ with the contraction group property. Then, for any compact open subgroup $L \le G$, there exists a Cartan-like decomposition $G=LUL$ with the contraction group property.
\end{prop}

\begin{proof}
Let $G$ be a topological group and let $K \le G$ and $G = KAK$ satisfy the contraction group property. 

Consider an open subgroup $K' \le K$. By compactness of $K$, there exist elements $g_1, \dots, g_n \in G$ such that $K = \bigsqcup_{i=1}^{n} g_i K'$. Then $G$ also decomposes as $G = K'A'K'$ where $A' := \bigcup_{i,j} g_{i}^{-1} A g_{j}$. Now suppose that $( h_i)_{i=1}^{\infty} \subseteq A'$ is a sequence of coset representatives. We need to show that the sequence $(h_i)_{i=1}^{\infty}$ is either bounded or has a subsequence with non-trivial contraction group.

If the sequence is bounded, we are done, so we may suppose that it is unbounded and show that it has a subsequence with non-trivial contraction group. Since there are only finitely many $g_i$, by passing to a subsequence of the $h_i$, we may suppose that for each $i$, $h_i = g_{l}^{-1} a_i g_{m}$ with $a_i \in A$ and fixed $l,m \in \{ 1, \dots, n \}$. Since $(h_i)_{i=1}^{\infty}$ is unbounded, so must be the sequence $(a_i)_{i=1}^{\infty}$, and hence by assumption, there exists a subsequence $(a_{i_j})_{j=1}^{\infty} \subseteq (a_i)_{i=1}^{\infty}$ and a non-trivial $x \in \con((a_{i_j})_{j=1}^{\infty})$. Then, $\tilde{x} = g_{l}^{-1} x g_{m}$ is a non-trivial element in the contraction group of the subsequence $(h_{i_j})_{j=1}^{\infty}$. Thus the decomposition $G = K'A'K'$ has the contraction group property.

On the other hand, suppose that $K'$ is a compact open subgroup of $G$ containing $K$. Since $K \le K'$, there exists a Cartan-like decomposition $G = K' A' K'$ with $A' \subseteq A$. It then follows that the decomposition $G = K' A' K'$ has the contraction group property since the decomposition $G = KAK$ does.

We have thus shown that if a compact open subgroup $K'$ of $G$ contains $K$ or is contained in $K$, then there exists a Cartan-like decomposition of $G$ with respect to $K'$ satisfying the contraction group property. Now suppose that $L$ is an arbitrary compact open subgroup not necessarily contained in or containing $K$. Then $L \cap K$ is a compact open subgroup contained in $K$, hence there exists a Cartan-like decomposition of $G$ with respect to $L \cap K$ satisfying the contraction group property. Since $L$ contains $L \cap K$, it follows that there is a Cartan-like decomposition of $G$ with respect to $L$ satisfying the contraction group property.
\end{proof}

Ideally, the contraction group property would also not depend on the choice of coset representatives in a Cartan-like decomposition, however, unfortunately, this is not always the case as illustrated in the following example.

\begin{exm}
Let $G = \PGL(\Qp)$ and $K = \PGL(\Zp)$. Note that $K$ is a compact open subgroup of $G$ and $G = \bigsqcup_{n \in \mathbb{N}} K g^n K$ where $g = (\begin{smallmatrix} p & 0 \\ 0 & 1 \end{smallmatrix})$. It can be checked that $(\begin{smallmatrix} 1 & 1 \\ 0 & 1 \end{smallmatrix}) \in \con((g^n)_{n=1}^{\infty})$. 

Since $k_n = (\begin{smallmatrix} 1 & 0 \\ p^{\lceil n/3 \rceil} & 1 \end{smallmatrix})$ belongs to $K$ for each $n \in \mathbb{N}$, the elements 
$$h_n := g^n k_n = (\begin{smallmatrix} p^n & 0 \\ p^{\lceil n/3 \rceil} & 1 \end{smallmatrix})$$
\noindent also form a set of coset representatives for a Cartan-like decomposition of $G$ with respect to $K$. We claim that the contraction group of every subsequence of the sequence $(h_n)_{n=1}^{\infty}$ is trivial.

Suppose that $h = (\begin{smallmatrix} a & b \\ c & d \end{smallmatrix}) \in G$ and  $h \in \con((h_n)_{n=1}^{\infty})$. Then, it may be checked that
$$h_n h h_{n}^{-1} = \begin{pmatrix} a - bp^{\lceil n/3 \rceil} & bp^{n} \\ (a - d)p^{\lceil n/3 \rceil - n} + cp^{-n} - bp^{2 \lceil n/3 \rceil -n} & bp^{\lceil n/3 \rceil} + d \end{pmatrix}.$$
\noindent Since it is assumed that $h_nhh_n^{-1} \rightarrow (\begin{smallmatrix} 1 & 0 \\ 0 & 1 \end{smallmatrix})$, we must have $b = c = 0$ and $a=d$, otherwise, the norm of the bottom left entry of the matrix will diverge to $\infty$. It then follows that $h$ is the identity in $G$ and $\con((h_n)_{n=1}^{\infty})$ is trivial. The same argument applies to every subsequence of $(h_n)_{n=1}^{\infty}$. 
\end{exm}

This demonstrates that if one Cartan-like decomposition of a group has the contraction group property, then another Cartan-like decomposition may not, even if we keep the same compact open subgroup. For some groups, however, the contraction group property is independent of the choice of coset representatives.

\begin{prop}\label{prop:prop1}
Let $\T$ be an infinite locally finite tree without leaves and $G \le \Aut(\T)$. Suppose that the fixator of any half-tree of $\T$ in $G$ is non-trivial. Then every unbounded sequence $(g_i)_{i=1}^{\infty}\subset G$ has a subsequence with non-trivial contraction group.
\end{prop}

\begin{proof}
Fix a vertex $v \in V\T$. Since the sequence $(g_i)_{i=1}^{\infty}$ is unbounded, we may assume, by passing to a subsequence if necessary, that for each $i \geq 1$ the distance from $v$ to $g_i(v)$ is at least $i$ and, since $\T$ is locally finite, that the neighbour of $v$ on the path from $v$ to $g_i(v)$ is always the same vertex, $w \in V\T$ say. 

If infinitely many of the $g_i$ are translations with $v$ on the axis, we may suppose that they all are by passing to a subsequence. Then $w$ is on the axis of each of the $g_i$ too, and $g_i$ translates $\T_{(w,v)}$ to the half-tree $\T_{(g_i(w),g_i(v))}$, which contains the ball with centre $v$ and radius $i$. Choose $x \in G$ to fix $\T_{(w,v)}$ and act non-trivially on $\T_{(v,w)}$, which exists by assumption. Since $g_i x g_i^{-1}$ fixes the ball of radius $i$ around $v$ for each $i$, $g_i x g_i^{-1} \to \id$. Hence $x$ is a non-trivial element in the contraction group of a subsequence of the $g_i$ and we are done.

If only finitely many of the $g_i$ are translations with $v$ on their axis, then it may be assumed that no $a_i$ is a translation with $v$ on its axis. Then each of the $a_i$ are either elliptic elements or translations with $v$ not on the axis. Choose $x \in G$ that fixes $\T_{(v,w)}$ and acts non-trivially on $\T_{(w,v)}$. Then $x$ fixes the ball of radius $i$ around $g_i(v)$ for each $i$. Hence $g_i x g^{-1}_i$ fixes the ball of radius $i$ around $v$ for each $i$ and converges to the identity. Therefore $x$ is a non-trivial element of the contraction group of a subsequence of the $g_i$, which completes the proof.
\end{proof}

\begin{rem}
Note that any group satisfying the hypotheses of the above proposition is necessarily non-discrete. 
\end{rem}

\begin{cor}
Let $\T$ be an infinite locally finite tree without leaves and $G \le \Aut(\T)$. Suppose that the fixator of any half-tree of $\T$ in $G$ is non-trivial. Then every Cartan-like decomposition of $G$ has the contraction group property.
\end{cor}

\begin{proof}
Follows directly from the previous proposition.
\end{proof}


\subsection{The Closed Range Property}

This section will study some general results concerning the closed range property. First we prove a generalisation of Theorem 4.2 from \cite{CW20}. We also show that a sufficient condition for a locally compact group to have the closed range property is that it has a cocompact subgroup with the closed range property.

\begin{thm}\label{thm:closedrangethm}
Let $G$ be a topologically simple group that has the contraction group property. Then $G$ has the closed range property.
\end{thm}

\begin{proof}
Suppose the hypotheses of the theorem. Let $K \le G$ be a compact open subgroup such that there exists a Cartan-like decomposition $G = KAK$ with the contraction group property. Suppose that $\varphi: G \rightarrow H$ is a non-trivial continuous homomorphism to an arbitrary topological group $H$. Consider a sequence $(g_i)_{i=1}^{\infty} \subseteq G$ and suppose that $\varphi(g_i)$ converges to $h \in H$. It must be shown that $h \in \varphi(G)$. Now, there are sequences $(k_i)_{i=1}^{\infty},\ (k'_i)_{i=1}^{\infty}\subseteq K$ and $(a_i)_{i=1}^{\infty} \subseteq A$ such that $g_{i} = k_{i} a_{i} k'_{i}$ for each $i$. Passing to a subsequence if necessary, we may suppose, by compactness of $K$, that the sequences $(k_{i})_{i=1}^{\infty}$ and $(k'_{i})_{i=1}^{\infty}$ converge to elements $k, k' \in K$ respectively. Then $\varphi(a_{i}) = \varphi(k_{i})^{-1} \varphi(g_i) \varphi(k'_{i})^{-1} \to \varphi(k)^{-1} h \varphi(k')^{-1}\mbox{ as }i \to \infty.$ Thus the sequence $(\varphi(a_{i}))_{i=1}^\infty$ converges.

If the sequence $(a_{i})_{i=1}^{\infty}$ is bounded, it may be supposed, by passing to a subsequence if necessary, that the sequence is constant. Then $a_{i} = a \in A$ for each $i$ and $h = \varphi(k) \varphi(a) \varphi(k') \in \varphi(G)$. Thus the proof is complete. Suppose that the sequence $(a_i)_{i=1}^{\infty}$ is unbounded and set $\hat{a} := \lim_{i \to \infty} \varphi(a_{i})$. By assumption, there exists a subsequence $(a_{i_j})_{j=1}^{\infty} \subseteq (a_i)_{i=1}^{\infty}$ and a non-trivial $x \in \con((a_{i_j})_{j=1}^{\infty})$. Then $\hat{a} \varphi(x) \hat{a}^{-1} = \lim_{i \to \infty} \varphi(a_{i} x a_{i}^{-1}) = \lim_{j \to \infty} \varphi(a_{i_j} x a_{i_j}^{-1}) = \id_H$ by definition of the contraction subgroup for the sequence $(a_{i_j})_{j=1}^{\infty}$ and continuity of $\varphi$. Hence the kernel of $\varphi$ contains $\con((a_{i_j})_{i=1}^{\infty})$ and $\varphi$ must be the trivial homomorphism because $G$ is topologically simple, a contradiction. This completes the proof.
\end{proof}

Recall that given a topological group $G$ and a subgroup $H \le G$, $H$ is said to be \textit{cocompact} in $G$ if its quotient $G/H$ is compact. We now show that if a locally compact group $G$ has a cocompact subgroup with the closed range property, then $G$ also has the closed range property. We need the following two lemmas for the proof.

\begin{lem}\label{lem:ABcpt}
Let $G$ be a topological group and $A,B \subseteq G$ with $A$ compact and $B$ closed. Then the set $AB$ is closed.
\end{lem}

The proof of this lemma can be found for instance in \cite[Theorem 4.4]{HR79}. We now prove the following lemma about locally compact groups.

\begin{lem}
\label{lem:cocompact}
Let $G$ be a locally compact group and suppose that $H$ is a cocompact subgroup of $G$. Then there exists a compact set $K \subseteq G$ such that $G = KH$.
\end{lem}

\begin{proof}
Let $\mathcal{U}$ be the collection of open sets in $G$ with compact closure and let $\pi:G \rightarrow G/H$ the canonical map. Since $\pi$ is an open map, $\{ \pi(U) | U \in \mathcal{U} \}$ forms an open covering of $G/H$. By compactness of $G/H$, there is a finite subcover say $\pi(U_{1}), \dots, \pi(U_{n})$ of $G/H$. Then it follows that $K = \bigcup_{i=1}^{n} \overline{U_{i}}$ is a compact subset of $G$, being a finite union of compact sets, and $\pi(K) = KH = G$ by construction.
\end{proof}

This leads us to the main result:

\begin{prop}\label{prop:cocomclosed}
Let $G$ be a locally compact group. Suppose that $H$ is a cocompact subgroup of $G$ with the closed range property. Then $G$ also has the closed range property.
\end{prop}

\begin{proof}
By Lemma~\ref{lem:cocompact}, there is a compact subset $K$ of $G$ such that $G = KH$. Let $\varphi: G \rightarrow L$ be a continuous homomorphism to an arbitrary topological group $L$. By assumption, we know that $\varphi(H)$ is closed in $L$, since $H$ satisfies the closed range property and the restriction of $\varphi$ to $H$ is continuous. Since $\varphi$ is continuous, $\varphi(K)$ is compact in $L$, and then by Lemma \ref{lem:ABcpt}, $\varphi(G) = \varphi(K) \varphi(H)$ is closed being the product of a compact set and a closed set.
\end{proof}


\section{Homomorphic Images of Tree Automorphism Groups}\label{chap:5}

In this section we study the contraction group and closed range properties of groups of automorphisms of infinite locally-finite trees. We extend the work on Cartan-like decompositions and the closed range property seen in \cite{CW20} to a larger class of groups acting on trees. Recall the definition of $G^{+_k}$ defined in the preliminaries section of the article. It is shown that, under standard assumptions, for any closed subgroup $G \le \Aut(\T)$, the group $G^{+_k}$ has the closed range property.  This leads to some closed range results for the (generalised) universal groups and groups acting on trees with a locally semiprimitive action. 

Throughout this section we assume that $\T$ is an arbitrary infinite locally finite tree without leaves. First we note that closed subgroups $G \le \Aut(\T)$ admit a Cartan-like decomposition with a vertex stabiliser as the compact open subgroup.

\begin{prop}\label{prop:kak}
Let $G \le \Aut(\T)$ be a closed subgroup and for a fixed $v \in V\T$ let $K := G_v$. Then $K$ is a compact open subgroup of $G$ and $G$ admits a Cartan-like decomposition $G = KAK$ where $A$ is in one-to-one correspondence with the orbits of $K$ acting on $Gv \cap S(v,n)$ for each $n \in \mathbb{N}$.
\end{prop}

\begin{proof}
Fix $v \in V\T$. It is clear that $K = G_v$ is a compact open subgroup in the subspace topology on $G$, being the intersection of the compact open subgroup $\Aut(\T)_v$ with $G$. 

To show that $G$ admits a Cartan-like decomposition of the given form, enumerate the orbits of $K$ acting on the spheres $S(v,n) \cap Gv \subseteq \T$ for each $n \in \mathbb{N}$. For each orbit, choose a vertex $w$ in that orbit and an automorphism $a_w \in G$ that sends $v$ to $w$. Let $A$ be the collection of all the chosen $a_w$. We claim that $G = KAK$. Indeed, let $g \in G$.  There exists a vertex $w \in V\T$ in the $K$-orbit of $g(v)$ and an automorphism $a_w \in A$ that sends $v$ to $w$. Let $k \in K$ such that $kg(v) = w$. Then $a_{w}^{-1}kg(v) = v$, so $a_{w}^{-1}kg \in K$, thus $g \in Ka_{w}K$. It follows that $G = KAK$. 

We now show that the double cosets $KaK$ for distinct $a \in A$ are disjoint {\it i.e.\/} that $A$ is a set of coset representatives for the decomposition. Suppose that there exists $a_u, a_w \in A$ such that $Ka_u K = Ka_w K$. Then there exists $k_1, k_2 \in K$ such that $a_u = k_1 a_w k_2$. By definition of $a_u$ and $a_w$, $k_1$ maps $w$ to $u$, and so $u, w \in V\T$ are in the same $K$-orbit. Hence $a_u = a_w$ by construction of the set $A$.
\end{proof}

%
%

For use in the forthcoming theorem, we need the following three lemmas. The first lemma is Lemma 4 in \cite{MV12} restated for use here.

\begin{lem}\label{lem:moller}
\cite[Lemma 4]{MV12} Suppose that $G \le \Aut(\T)$ does not stabilise any non-empty proper subtree of $\T$. Then the following hold:
\begin{enumerate}[(i)]
   \item Suppose that there is some edge $\{u,v\} \in E\T$ such that the pointwise stabilisers of the half-trees $\T_{(u,v)}$ and $\T_{(v,u)}$ are non-trivial. Then the pointwise stabiliser of every half-tree in $\T$ is non-trivial.\
   \item Suppose that there is some edge $\{u,v\} \in E\T$ such that the pointwise stabiliser of $\T_{(u,v)}$ is trivial while the pointwise stabiliser of $\T_{(v,u)}$ is non-trivial. Then $G$ must fix an end of $\T$. 
\end{enumerate}
\end{lem}

The following result of Tits \cite[L\'emm\`e 4.4]{Tits70} will also be needed:

\begin{lem}\label{lem:titslem1}
\cite[L\'emm\`e 4.4]{Tits70} Suppose that $N$ and $G$ are non-trivial subgroups of $\Aut(\T)$ and $N$ is normalised by $G$. If $G$ does not stabilise any proper non-empty subtree or fix any end of $\T$, then the same is true for $N$.
\end{lem}

We prove the following lemma which is a consequence of the previous two. We note that Property $P_k$ is defined and investigated in \cite{BEW15}.

\begin{lem}
Suppose that $G \le \Aut(\T)$ and assume that $G$ does not stabilise any proper non-empty subtree or fix an end of $\T$, and satisfies Property $P_k$. If $G^{+_k}$ is non-trivial, then the fixator in $G^{+_k}$ of every half-tree in $\T$ is non-trivial.
\end{lem}

\begin{proof}
Since $G^{+_k}$ is normal in $G$, by Lemma \ref{lem:titslem1}, $G^{+_k}$ does not stabilise any proper non-empty subtree or end of $\T$. Since $G^{+_k}$ is non-trivial, there exists an edge $e = \{v,w\} \in E\T$ and a non-trivial element $g \in F_{k,e} = \Fix_G(B(v,k) \cap B(w,k))$. Now we know that $F_{k,e} = \Fix_{F_{k,e}}(\T_{(v,w)}) \Fix_{F_{k,e}}(\T_{(w,v)})$ since $G$ satisfies Property $P_k$. Thus, since $F_{k,e}$ is non-trivial, there must exist a non-trivial element $g'$ in either $\Fix_{F_{k,e}}(\T_{(w,v)})$ or $\Fix_{F_{k,e}}(\T_{(v,w)})$. Clearly $g' \in G^{+_k}$. Since $G^{+_k}$ does not stabilise any non-empty subtree or fix an end of $\T$, an application of Lemma \ref{lem:moller}$(ii)$ followed by an application of Lemma \ref{lem:moller}$(i)$ then shows that the stabiliser in $G^{+_k}$ of every half-tree in $\T$ must be non-trivial.
\end{proof}
 
 We come to the following theorem which shows that for closed subgroups $G \le \Aut(\T)$, the groups $G^{+_k}$ have the closed range property under certain assumptions.

\begin{thm}\label{thm:mainthm}
Let $G \le \Aut(\T)$  be a closed subgroup and suppose that $G$ does not stabilise any proper non-empty subtree or fix an end of $\T$. If $G$ satisfies Property $P_k$, then $G^{+_k}$ has the closed range property.
\end{thm}

\begin{proof}
First we note that $G^{+_k}$ is open in $G$ since it contains the open neighbourhood $\mathcal{U}(\id,B(v,k) \cap B(w,k))$ of the identity, where $\{v,w\} \in E\T$. Since $G^{+_k}$ is open in $G$, it is also closed in $G$, and since $G$ is closed in $\Aut(\T)$, it follows that $G^{+_k}$ is closed in $\Aut(\T)$. By Proposition \ref{prop:kak}, $G^{+_k}$ admits a Cartan-like decomposition $G^{+_k} = KAK$, with $K = G^{+_k}_v$ and $A \subseteq G^{+_k}$ as constructed in the proposition. 

We also know that $G^{+_k}$ is either trivial or simple by \cite[Theorem 7.3]{BEW15}. If $G^{+_k}$ is trivial, then $G^{+_k}$ clearly satisfies the closed range property, so we may suppose that $G^{+_k}$ is non-trivial and simple. By Theorem \ref{thm:closedrangethm}, we just need to show that the Cartan-like decomposition $G^{+_k} = KAK$ has the contraction group property. This follows from an application of the previous lemma followed by an application of Proposition \ref{prop:prop1}.
\end{proof}

We now state a number of corollaries that result from this theorem.

\begin{cor}
Let $G \le \Aut(\T)$ and suppose that $G$ does not fix any proper non-empty subtree or fix an end of $\T$. Then $(G^{(k)})^{+_k}$ has the closed range property.
\end{cor}

\begin{proof}
Since $G^{(k)}$ contains $G$, $G^{(k)}$ does not fix any non-empty subtree or end of $\T$. Also, $G^{(k)}$ is closed by \cite[Proposition 3.4]{BEW15}, and satisfies Property $P_k$ by combining Proposition 5.2 and Corollary 6.4 of \cite{BEW15}. Now apply the previous theorem.
\end{proof}

Next, using the notation of universal groups as set out in \cite{Tor20}, we prove the following two results. Since the generalised universal groups $U_{k}(F)$ satisfy Property $P_k$ \cite[Proposition 3.7]{Tor20} and do not stabilise any proper non-empty subtree or fix an end of $\T$, this also gives us the following.

\begin{cor}
\label{cor:Ud}
Let $F \le \Aut(B_{d,k})$. Then $U_{k}(F)^{+_k}$ satisfies the closed range property.
\end{cor}

The theorem also allows us to show that the universal groups $U_{1}(\F)$ satisfy the closed range property when $\F$ is transitive and generated by point stabilisers.

\begin{cor}
\label{cor:U1}
Let $\F \le \Sym(d)$. Then $U_{1}(\F)^{+_1}$ has the closed range property. Moreover, if $\F$ is transitive and generated by point stabilisers, then $U_{1}(\F)$ has the closed range property.
\end{cor}

\begin{proof}
That $U_{1}(\F)^{+_1}$ has the closed range property is just a special case of the previous corollary. When $\F$ is transitive and generated by point stabilisers, $U_{1}(\F)^{+_1}$ has index 2 in $U_{1}(\F)$ by \cite[Proposition 3.2.1]{BM00}, and an application of Proposition \ref{prop:cocomclosed} shows that $U_{1}(\F)$ has the closed range property.
\end{proof}

\begin{rem}
The hypotheses of Corollary~\ref{cor:U1} imply that $U_{1}(\F)^{+_1}$ is non-trivial. On the other hand, the group $U_{k}(F)$ in Corollary~\ref{cor:Ud} may be discrete, see \cite[Proposition~3.12]{Tor20}, and so it may happen that $U_{k}(F)^{+_k}$ is trivial.
\end{rem}

For use in the following corollary, a group $G \le \Aut(\T)$ is \textit{locally semi-primitive} if for every $v \in V\T$, the vertex stabiliser $G_v$ acts as a semi-primitive permutation group on the edges incident to $v$ in $\T$. A permutation group is \textit{semi-primitive} if it is transitive and all its normal subgroups are either transitive or semi-regular \cite{GM18}. (A permutation group is \emph{semi-regular} if the identity is the only element that fixes a point, which is equivalent to saying that the action of the group is \emph{free}. This is not to be confused with a building being semi-regular.)

\begin{cor}
Let $G \le \Aut(\T)$ be closed, non-discrete and locally semi-primitive. If $G$ does not fix any proper non-empty subtree or end of $\T$, and satisfies Property $P_k$, then $G$ has the closed range property.
\end{cor}

\begin{proof}
By the theorem, $G^{+_k}$ has the closed range property, and by \cite[Proposition 2.11(iii)]{Tor20}, $G^{+_k}$ is cocompact in $G$ since it is a normal subgroup of $G$. An application of Proposition \ref{prop:cocomclosed} shows that $G$ also has the closed range property.
\end{proof}

We now provide an example of a class of totally disconnected locally compact groups acting on trees that do not have the closed range property.

\begin{exm}
For $d \in \mathbb{N}_{\ge2}$ and any $\F \le \FF \le \Sym(d)$, Le Boudec defines in \cite{LB16} a class of groups denoted $G(\F,\FF)$ that act on the $d$-regular tree and are a generalisation of the Universal groups mentioned above. It can be easily deduced from \cite[Corollary 3.6]{LB16} that whenever $\F \lneq \FF$ (equivalently $G(\F, \FF) \lneq U(\FF)$), the inclusion map $G(\F, \FF) \hookrightarrow \Aut(\Tq)$ is continuous but not closed.
\end{exm}


\section{Homomorphic Images of Automorphism Groups of Buildings}

In this section, we study Cartan-like decompositions of automorphism groups of semi-regular right-angled buildings and prove a closed range result for these groups.  We fix a locally-finite semi-regular right-angled building $\Delta$ of type $(W,S)$. We start out by proving that the group $\Aut(\Delta)^{+}$ of type preserving automorphisms of $\Delta$ admits a Cartan-like decomposition with the coset representatives chosen in one-to-one correspondence with the elements of $W$.

By Proposition 6.1 in \cite{Cap14}, the group $\Aut(\Delta)^{+}$ acts strongly transitively on $\Delta$ {\it i.e.\/} $\Aut(\Delta)^{+}$ acts transitively on pairs $(\mathcal{A}, C)$ where $\mathcal{A}$ is an apartment of $\Delta$ and $C$ is a chamber of $\mathcal{A}$. For a fixed chamber $C \in \Ch(\Delta)$, the group of automorphisms in $\Aut(\Delta)^{+}$ that fix $C$, denoted $\Aut(\Delta)^{+}_{C}$, is a compact open subgroup of $\Aut(\Delta)^{+}$ in the permutation topology. The following proposition gives an enumeration of the coset representatives for a Cartan-like decomposition of $\Aut(\Delta)^{+}$ over $\Aut(\Delta)^{+}_{C}$.

\begin{prop} \label{prop:buildcartan}
Let $\Delta$ be a locally-finite semi-regular right-angled building of type $(W,S)$. Fix $C \in \Ch(\Delta)$, an apartment $\mathcal{A}$ in $\Delta$ containing $C$ and let $K := \Aut(\Delta)^{+}_{C}$. For each $w \in W$, choose an automorphism $a_w \in \Aut(\Delta)^{+}$ stabilising $\mathcal{A}$ and sending $C$ to the chamber $C'$ of $\mathcal{A}$ at Weyl distance $w$ from $C$. Set $A :=\{ a_w \mid w \in W \}$. Then $\Aut(\Delta)^{+} = KAK$ is a Cartan-like decomposition.
\end{prop}

\begin{proof}
Let $g \in \Aut(\Delta)^{+}$ and suppose that $\delta(C, g(C)) = w$. Choose an apartment $\mathcal{A}'$ containing $C$ and $g(C)$. Then, by strong transitivity of $\Aut(\Delta)^{+}$, there exists an automorphism $k \in K$ mapping the pair $(\mathcal{A}',C)$ to the pair $(\mathcal{A},C)$. Since $k$ is type-preserving, $\delta(C, kg(C)) = w$ and hence it follows that $a_{w}^{-1}kg(C) = C$. Thus there exists $k' \in K$ with $a_{w}^{-1}kg = k'$ {\it i.e.\/} $g = k^{-1}a_{w}k' \in KAK$.

We now show that the double cosets $KaK$ for distinct $a \in A$ are disjoint {\it i.e.\/} that $A$ is a set of coset representatives for the decomposition. Suppose that there exists $a_u, a_w \in A$ such that $Ka_u K = Ka_w K$. Then there exists $k_1, k_2 \in K$ such that $a_u = k_1 a_w k_2$. By definition of $a_u$ and $a_w$, and since $k_1$ and $k_2$ are type preserving and fix $C$, we have $w=\delta(C, a_wC) = \delta(C, k_1 a_w k_2 C) = \delta(C,a_uC) = u$. Thus $a_u = a_w$ by construction. This completes the proof.
\end{proof}

\begin{rem}
Any group with a Weyl transitive action on a building, in particular any group with a BN-pair, admits a decomposition of this form with the coset representatives indexed by the elements of the Weyl group $W$ (see \cite[Chapter 6]{AB08}). This decomposition is usually called a \textit{Bruhat} decomposition in the theory of buildings. Note that $K$ is not necessarily compact when the building is not locally finite.
\end{rem}

We now prove that, under certain assumptions, the type-preserving automorphism groups of locally-finite semi-regular right-angled buildings along with the corresponding universal groups have the closed range property. We will more or less replicate the proof of \cite[Theorem 28]{CR09}.

First we introduce some notation. We follow the terminology and definitions used in \cite{Wei04,DMSS18, DMS19}. Given a chamber $C \in \Ch(\Delta)$ and $s \in S$, the \textit{$s$-wing} of $C$, denoted $X_s(C)$, is defined as $X_s(C):=\{ D \in \Ch(\Delta) \mid \proj_{\mathcal{P}_{s}(C)}(D) = C \}$. 

Fix an apartment $\mathcal{A}$ of $\Delta$. A \textit{reflection} of $\mathcal{A}$ is an element $r \in W$ that fixes an edge in $\mathcal{A}$. The set of edges fixed by $r$ is called the \textit{wall} of $r$. If $\{D_1,D_2\}$ is an edge in the wall of $r$, then the chambers of $\mathcal{A}$ split into two disjoint sets $\alpha := \{ C \in \Ch(\mathcal{A}) \mid d(C,D_1) < d(C,D_2) \}$ and its complement $-\alpha := \{ C \in \Ch(\mathcal{A}) \mid d(C,D_1) > d(C,D_2) \}$ called the \textit{roots} associated to $r$ \cite[Proposition 3.11]{Wei04}. It can be shown that these sets are independent of the choice of the edge $\{D_1,D_2\}$ in the wall of $r$ \cite[Corollary 3.15]{Wei04}. Given any root $\alpha$, the root $-\alpha$ is called the root \textit{opposite} $\alpha$. A root in the building $\Delta$ is defined to be a root in one of its apartments. The set of roots of an apartment $\mathcal{A}$ in $\Delta$ will be denoted by $\Phi(\mathcal{A})$. 

Given any root $\alpha$ in $\Delta$, the wall of $\alpha$ consists of all panels with chambers in both $\alpha$ and $-\alpha$ \cite[Proposition 3.12]{Wei04}. Since $\Delta$ is right-angled, all panels in the wall of $\Delta$ are of the same type. Thus we may define the type of a root $\alpha$ to be the type of the panels in its wall. The $s$-wings in $\Delta$ are precisely the roots of type $s$. The \textit{root wing group} associated to a root $\alpha$ of type $s$ is defined as:
\begin{displaymath} U_{\alpha} = \Fix_{\Aut(\Delta)^{+}}(X_s(C)) \end{displaymath}

\noindent for some chamber $C$ in the wall of $\alpha$. The definition does not depend on the choice of $C$ (see \cite[Section 3 \& Section 4]{DMS19} for more details). The following result will be useful in the proof of the following theorem.

\begin{lem}\label{lem:lem3}
\cite[Lemma 4.10]{DMS19} Fix a chamber $C \in \Ch(\Delta)$, an apartment $\mathcal{A}$ containing $C$ and $n \in \mathbb{N}$. For every $\alpha \in \Phi(\mathcal{A})$ with $d(C, \alpha) > n$, the group $U_{- \alpha}$ is contained in $\Fix_{\Aut(\Delta)^{+}}(B(C,n))$.
\end{lem}

We also need the following lemma, which applies to any irreducible non-spherical Coxeter system.

\begin{lem}\label{lem:lem4}
Let $(W,S)$ be an irreducible non-spherical Coxeter system with $S$ finite, and let $\Delta$ be the corresponding thin building. Let $C$ be a chamber in $\Delta$. Suppose that $\mathcal{W}$ is an infinite subset of $W$. Then there is a root $\alpha \in \Phi(\Delta)$, associated to an $s\in S$, and a sequence $(w_k)_{k=1}^{\infty} \subseteq \mathcal{W}$ such that $d(C, w_k(\alpha)) \rightarrow \infty$. 
\end{lem}

\begin{proof}
This is essentially \cite[Lemma 27]{CR09} which is given a geometric proof there. To make this article more self-contained, we give an alternative proof that is algebraic. 

It may be supposed without loss of generality that $C$ is the fundamental chamber corresponding to the identity in $W$. Then $\{C,sC\}$ is a panel fixed by $s\in S$. Denote the wall containing that panel by $\omega_s$ and the corresponding roots by $\alpha_s$ and $-\alpha_s$. There is $D\in\omega_s$ such that $d(wC,D)$ is equal to the distance from $wC$ to $\omega_s$, which is equal to $d(swC,sD)$ and to the distance from $swC$ to $\omega_s$. Hence $2d(wC,D) +1= d(wC,swC) = \ell(w^{-1}sw)$, see \cite[Equation 3.7]{AB08}. 
\begin{claim*} Let $R>0$. Then the cardinality of $\{w\in W\mid \ell(w^{-1}sw)\leq R \text{ for all }s\in S\}$ is finite.
\end{claim*}
\begin{proof} Suppose otherwise. Then, since $W$ is infinite and $S$ finite, there are $w_1\ne w_2$ in $W$ such that 
$$
w_1^{-1}sw_1 = w_2^{-1}sw_2 \text{ for all }s\in S.
$$
Hence $w_1w_2^{-1}\ne\id$ and belongs to the centre of $W$, which contradicts that the centre of $W$ is trivial \cite[Proposition 2.37]{AB08}.
\end{proof}

Let $\mathcal{W}$ be the infinite subset of $W$ in the hypothesis. Then the claim and the paragraph that precedes it, and again using that $S$ is finite, imply that there is $s\in S$ and a sequence $(w_k)_{k=1}^{\infty}$ in $\mathcal{W}$ such that the distance from $w_k^{-1}$ to $\omega_s$ diverges to infinity as $k\to\infty$. Since $\alpha_s$ and $-\alpha_s$ partition $W\setminus \omega_s$ it may be supposed that $w_k^{-1}\in-\alpha_s$ for all $k$. Then $d(w_k^{-1}C,\alpha)\to\infty$ and $k\to\infty$ and the proof is complete because $d(w_k^{-1}C,\alpha)=d(C,w_k(\alpha))$.
\end{proof}

Note for use in the following proof, when $\Delta$ is a semi-regular right-angled building of type $(W,S)$, and $(W,S)$ is irreducible and non-spherical, then $\Aut(\Delta)^{+}$ is simple \cite[Theorem 1.1]{Cap14}.

\begin{thm}\label{thm:mainthm2}
Let $\Delta$ be a locally-finite semi-regular right-angled building of type $(W,S)$. Assume that $(W,S)$ is irreducible and non-spherical. Then the group $\Aut(\Delta)^{+}$ has the closed range property.
\end{thm}

\begin{proof}
Let $\mathcal{A}$ be an apartment in $\Delta$ and decompose $\Aut(\Delta)^{+}$ as $\Aut(\Delta)^{+} = KAK$ as in Proposition \ref{prop:buildcartan}. By simplicity of $\Aut(\Delta)^{+}$ and Theorem \ref{thm:closedrangethm}, we just need to show that this decomposition satisfies the contraction group property. 

Let $(a_{w_i})_{i=1}^{\infty}$ be a sequence in $A$. If the sequence is bounded we are done. So suppose that the sequence is unbounded, in which case $\{w_i\}_{i=1}^\infty$ is infinite. By Lemma \ref{lem:lem4}, and because $\mathcal{A}$ is convex, there exists a root $\alpha \in \Phi(\mathcal{A})$ and a subsequence $(a_{w_{i_k}})_{k=1}^{\infty} \subseteq (a_{w_i})_{i=1}^{\infty}$ such that $d(C, a_{w_{i_k}} (\alpha)) \rightarrow \infty$. Choose a non-trivial element $x \in U_{-\alpha}$. Then $a_{w_{i_k}}xa_{w_{i_k}}^{-1} \in U_{-a_{w_{i_k}}(\alpha)}$ for all $k$. By Lemma \ref{lem:lem3}, we must have that $a_{w_{i_k}}xa_{w_{i_k}}^{-1} \rightarrow \id$ and hence $x$ is a non-trivial element in the contraction group of the $a_{w_{i_k}}$. This completes the proof.
\end{proof}

Adopting the notation from the paper \cite{DMSS18} for universal groups of semi-regular right-angled buildings, and using the simplicity result \cite[Theorem 3.25]{DMSS18}, an almost identical proof gives the following.

\begin{cor}
Let $\Delta$ be a thick locally-finite semi-regular right-angled building of irreducible type $(W,S)$ and rank at least 2. For each $s \in S$, let $h_s: \Ch(\Delta) \rightarrow Y_s$ be a legal $s$-colouring and $G^s \le \Sym(Y_s)$ be transitive and generated by point stabilisers. Then the universal group $U((G^s)_{s \in S})$ satisfies the closed range property. 
\end{cor}


\bibliographystyle{amsalpha}
\bibliography{Cartan_decomp_bib}

\providecommand{\bysame}{\leavevmode\hbox to3em{\hrulefill}\thinspace}
\providecommand{\MR}{\relax\ifhmode\unskip\space\fi MR }
\providecommand{\MRhref}[2]{%
  \href{http://www.ams.org/mathscinet-getitem?mr=#1}{#2}
}
\providecommand{\href}[2]{#2}
\begin{thebibliography}{DMSS18}

\bibitem[AB08]{AB08}
P.~Abramenko and K.~Brown, \emph{Buildings: {T}heory and {A}pplications},
  Graduate Texts in Mathematics, Springer-Verlag, 2008.

\bibitem[Abe74]{Abe73}
H.~Abels, \emph{Specker-{K}ompaktifizierungen von lokal kompakten topologischen
  {G}ruppen}, Mathematische Zeitschrift \textbf{135} (1973/74), 325--362.

\bibitem[BEW15]{BEW15}
C.~Banks, M.~Elder, and G.~A. Willis, \emph{Simple groups of automorphisms of
  trees determined by their actions on finite subtrees}, Journal of Group
  Theory \textbf{18} (2015), no.~2, 235--261.

\bibitem[BG14]{BG14}
U.~Bader and T.~Gelander, \emph{Equicontinuous actions of semisimple groups},
  arXiv:1408.4217 (2014).

\bibitem[BM00]{BM00}
M.~Burger and S.~Mozes, \emph{Groups acting on trees: from local to global
  structure}, Publications Math{\'e}matiques de l'IH{\'E}S \textbf{92} (2000),
  113--150.

\bibitem[Cap14]{Cap14}
P.~Caprace, \emph{Automorphism groups of right-angled buildings: simplicity and
  local splittings}, Fundamenta Mathematicae \textbf{224} (2014), 17--51.

\bibitem[CR09]{CR09}
P.~Caprace and B.~R\'emy, \emph{Simplicity and superrigidity of twin building
  lattices}, Inventiones Mathematicae \textbf{176} (2009), 169--221.

\bibitem[CW21]{CW20}
M.~Carter and G.~Willis, \emph{Decomposition {T}heorems for {A}utomorphism
  {G}roups of {T}rees}, Bulletin of the Australian Mathematical Society
  \textbf{103} (2021), 104--112.

\bibitem[DMS19]{DMS19}
T.~De~Medts and A.~Silva, \emph{Open subgroups of the automorphism group of a
  right-angled building}, Geometriae Dedicata \textbf{203} (2019), 1--23.

\bibitem[DMSS18]{DMSS18}
T.~De~Medts, A.~Silva, and K.~Struyve, \emph{Universal {G}roups for
  {R}ight-{A}ngled {B}uildings}, Groups, Geometry and Dynamics \textbf{12}
  (2018), 231--287.

\bibitem[Gle51]{Gle51}
A.~M. Gleason, \emph{The structure of locally compact groups}, Duke Mathematics
  Journal \textbf{18} (1951), 85--104.

\bibitem[Gle52]{Gle52}
\bysame, \emph{Groups without small subgroups}, Annales of Mathematics
  \textbf{56} (1952), 193--212.

\bibitem[GM18]{GM18}
M.~Giudici and L.~Morgan, \emph{A theory of semiprimitive groups}, Journal of
  Algebra \textbf{503} (2018), 146--185.

\bibitem[HP03]{HP03}
F.~Haglund and F.~Paulin, \emph{Constructions arborescentes d’immeubles},
  Math. Ann. \textbf{325} (2003), 137--164.

\bibitem[HR79]{HR79}
R.~Hewitt and K.A. Ross, \emph{Abstract {H}armonic {A}nalysis {I}}, Grundlehren
  der mathematischen Wissenschaften, Springer-Verlag, 1979, Second Edition.

\bibitem[KM08]{KM08}
B.~Kr\"on and R.G. M\"oller, \emph{Analogues of {C}ayley graphs for topological
  groups}, Mathematische Zeitschrift \textbf{258} (2008), no.~637.

\bibitem[LB16]{LB16}
A.~Le~Boudec, \emph{Groups acting on trees with almost prescribed local
  action}, Comment. Math. Helv. \textbf{91} (2016), 253--293.

\bibitem[MV12]{MV12}
R.~M\"oller and J.~Vonk, \emph{Normal subgroups of groups acting on trees and
  automorphism groups of graphs}, Journal of Group Theory \textbf{15} (2012),
  831--850.

\bibitem[MZ52]{MZ52}
D.~Montgomery and L.~Zippin, \emph{Small subgroups of finite-dimensional
  groups}, Annales of Mathematics \textbf{56} (1952), 213--241.

\bibitem[Omo66]{Omo66}
H.~Omori, \emph{Homomorphic {I}mages of {L}ie {G}roups}, J. Math. Soc. Japan.
  \textbf{18} (1966), 97--117.

\bibitem[Tit70]{Tits70}
J.~Tits, \emph{Sur le groupe des automorphismes d'un arbre}, Essays on Topology
  and Related Topics \textbf{92} (1970), 188--211.

\bibitem[Tor20]{Tor20}
S.~Tornier, \emph{Groups {A}cting on {T}rees with {P}rescribed {L}ocal
  {A}ction}, arXiv:2002.09876 (2020).

\bibitem[vD31]{VD31}
D.~van Dantzig, \emph{Studi\"en over topologische algebra}, Ph.D. thesis,
  Rijksuniversiteit Groningen, 1931.

\bibitem[vD36]{VD36}
\bysame, \emph{Zur topologischen {A}lgebra {III}. {B}rouwersche und
  {C}antorsche {G}ruppen}, Compositio Mathematica \textbf{3} (1936), 408--426.

\bibitem[Wei04]{Wei04}
R.~Weiss, \emph{The {S}tructure of {S}pherical {B}uildings}, Annals of
  Mathematics Studies, Princeton University Press, 2004.

\end{thebibliography}


\end{document}